\documentclass[11pt]{article}
\usepackage{bbm}
\usepackage{mathrsfs}
\usepackage{amsmath}
\usepackage{amsfonts}
\usepackage{amsmath,amsthm}
\usepackage{amsmath,amssymb,amsthm,latexsym}
\usepackage{amscd}
\usepackage[all]{xy}
\usepackage{hyperref}

\newtheorem{theorem}{Theorem}[section]
\newtheorem{proposition}[theorem]{Proposition}
\newtheorem{definition}[theorem]{Definition}

\newtheorem{conjecture}[theorem]{Conjecture}

\theoremstyle{remark}
\newtheorem{remark}[theorem]{Remark}
\def\o{\omega}
\def\<{\langle}
\def\>{\rangle}

\def\d{{\rm d}}
\def\i{{\rm i}}

\begin{document}
\title{\bf{The Moebius geometry of Wintgen ideal submanifolds}}
%The conformal Gauss map of Wintgen ideal submanifolds and harmonic maps
\author {Xiang Ma \footnote{Xiang Ma, School of Mathematical Sciences, Peking University,
Beijing 100871, People's Republic of China.
e-mail: {\sf maxiang@math.pku.edu.cn}. Funded by NSFC project 11171004.}~~~~ Zhenxiao Xie \footnote{Zhenxiao Xie, School of Mathematical Sciences, Peking University,
Beijing 100871, People's Republic of China.
e-mail: {\sf xiezhenxiao@126.com}}}
\maketitle

\begin{abstract}
Wintgen ideal submanifolds in space forms are those ones attaining equality pointwise in the so-called DDVV inequality which relates the scalar curvature, the mean curvature and the scalar normal curvature. They are M\"obius invariant objects. The mean curvature sphere defines a conformal Gauss map into a Grassmann manifold. We show that any Wintgen ideal submanifold has a Riemannian submersion structure over a Riemann surface with the fibers being round spheres. Then the conformal Gauss map is shown to be a super-conformal and harmonic map from the underlying Riemann surface. Some of our previous results are surveyed in the final part.
\end{abstract}

\hspace{2mm}


{\bf MSC(2000):\hspace{2mm} 53A10, 53A30, 53C12, 53C42}

\section{Introduction}

Geometers are always interested in beautiful shapes. In many cases they arise as the extremal cases of certain geometrical inequalities. In particular, it would be desirable to find some universal inequality, whose equality case include many non-trivial examples. It would be more interesting if such objects are invariant under a suitable transformation group.

For submanifolds in real space forms, such a universal inequality has been found, called the DDVV inequality. The extremal case defines the Wintgen ideal submanifolds. These are invariant object under the M\"obius transformations; in particular, the study of them from the viewpoint of M\"obius geometry is the focus of this paper.

Recall that given a $m$-dimensional submanifold $M^m$ immersed in a real space form of dimension $m+p$ with constant sectional curvature $c$, at any point there holds
\begin{equation}\label{1.1}
\textit{The DDVV inequality:}~~~~~~K\leq c+||H||^2-K_N.
\end{equation}
Here $K=\frac{2}{m(m-1)}\sum\nolimits_{i<j}\langle  R(e_i,e_j)e_j,e_i\rangle$ is the normalized scalar curvature with respect to the induced metric on $M$,
$H$ is the mean curvature vector,
and $K_N=\frac{2}{m(m-1)}||R^{\perp}||$ is the normal scalar curvature.

This remarkable inequality attracts many geometers, because it relates the most important intrinsic and extrinsic quantities at one point of a submanifold, and it takes an incredibly general form, without restrictions on the dimension/codimension, or any additional geometrical or topological assumptions. It was first conjectured by De Smet, Dillen, Verstraelen and Vrancken \cite{DDVV} in 1999, and proved by Ge and Tang \cite{Ge} in 2008. (Lu gave an independent proof in \cite{Lu2}.)

After discovering the DDVV inequality, people became interested in the extremal case \cite{Dajczer3,DDVV,Lu1,Lu2}. Wintgen \cite{Wint} first proved this inequality for surfaces in $\mathbb{S}^4$, where the equality is attained exactly when the surfaces are \emph{super-conformal}. That means at any point of the surface, the curvature ellipse is a circle, or equivalently, the Hopf differential is an isotropic differential form. According to the suggestion of Chen and other ones \cite{Chen10,Pe}, we make the following definition.

\begin{definition}
A submanifold $M^m$ of dimension $m$ and codimension $p$ in a real space form is called a \emph{Wintgen ideal submanifold} if the equality is attained at every point of $M^m$ in the DDVV inequality \eqref{1.1} .
By the characterization of Ge and Tang in \cite{Ge}, this happens if, and only if, at every point $x\in M$ there exists an orthonormal basis $\{e_1,\cdots,e_m\}$ of the tangent space $T_xM^m$ and an orthonormal basis $\{n_1,\cdots,n_p\}$ of the normal space $T_x^{\bot}M^m$,
such that the shape operators $\{A_r,r=1,\cdots,p\}$ take the form as below:
\begin{equation}\label{form1}
A_1=
\begin{pmatrix}
\lambda_1 & \mu_0 & 0 & \cdots & 0\\
\mu_0 & \lambda_1 & 0 & \cdots & 0\\
0  & 0 & \lambda_1 & \cdots & 0\\
\vdots & \vdots & \vdots & \ddots & \vdots\\
0  & 0 & 0 & \cdots & \lambda_1
\end{pmatrix},~
A_2=
\begin{pmatrix}
\lambda_2\!+\!\mu_0 & 0 & 0 & \cdots & 0\\
0 & \lambda_2\!-\!\mu_0 & 0 & \cdots & 0\\
0  & 0 & \lambda_2 & \cdots & 0\\
\vdots & \vdots & \vdots & \ddots & \vdots\\
0  & 0 & 0 & \cdots & \lambda_2
\end{pmatrix},
\end{equation}
\[A_3=\lambda_3I_m,~~~~~ A_\sigma=0~~(\sigma\ge 4), \]
where $I_m$ is the identity matrix of order $m$.
\end{definition}

People have found abundant examples of Wintgen ideal submanifolds \cite{Bryant91,Dajczer1,Dajczer2,Dajczer3,Gu,Lu1,XLMW}.
It is interesting yet difficult to obtain a complete classification of them.

We emphasize that generally they should be classified up to M\"obius transformations, because Wintgen ideal is an M\"obius invariant property\footnote{It was first noticed by Dajczer and Tojeiro in \cite{Dajczer3}, based on an equivalent formulation of the DDVV inequality in \cite{Dillen}.}. This follows directly from \eqref{1.1} and the fact that up to a factor, the traceless part of the second fundamental form is M\"obius invariant. So the most suitable framework for the study of Wintgen ideal submanifolds is M\"{o}bius geometry. This research program has been carried out by us recently in \cite{LiTZ1,LiTZ2,LiTZ3,XLMW} under various additional assumptions. Besides giving a survey of these work, we will also report two new results on general Wintgen ideal submanifolds.

For any submanifold $M^m$ immersed in $\mathbb{S}^{m+p}$, we can define the \emph{mean curvature sphere} at one point $x\in M^m$. It is the unique $m$-dimensional round sphere tangent to $M^m$ at $x$ which also shares the same mean curvature vector with $M^m$ at $x$. As a well-known M\"obius invariant construction\footnote{The notion of the mean curvature sphere can be traced back to Blaschke \cite{Blaschke} in 1920s.}, the characterization above holds true for any other conformal metric of $\mathbb{S}^{m+p}$. Via the light-cone model, this codimension-$p$ sphere corresponds to a space-like $p$-space $\mathrm{Span}_{\mathbb{R}}\{\xi_1,\cdots,\xi_p\}$ in the Lorentz space $\mathbb{R}^{m+p+2}_1$. We call it \emph{the conformal Gauss map} \footnote{This is an analog to the work of Bryant \cite{Bryant84} and Ejiri \cite{Ejiri} on Willmore surfaces in $\mathbb{S}^n$.} into the real Grassmannian \[\Xi=\xi_1\wedge\cdots\wedge\xi_p\in
\mathrm{Gr}(p,{\mathbb{R}^{m+p+2}_1}).\]

The crucial observation is that the image $\Xi(M^m)$ degenerates to a 2-dimensional surface when $M^m$ is Wintgen ideal. Moreover, we have:

\begin{theorem}\label{thm-harmonic}
For a Wintgen ideal submanifold, the conformal Gauss map $\Xi$ factors as a projection map $\pi:M^m\to \overline{M}^2$ (which is a Riemannian submersion up to a constant), and a super-conformal harmonic map from a Riemann surface \[\Xi:\overline{M}^2\to\mathrm{Gr}(p,{\mathbb{R}^{m+p+2}_1}).\]
In other words, $\Xi(M^m)$ is a super-minimal surface $\overline{M}^2\subset\mathrm{Gr}(p,{\mathbb{R}^{m+p+2}_1})$ (endowed with the induced metric).
\end{theorem}

This result shows striking similarity with the celebrated characterization of Willmore surfaces by its conformal Gauss map being a harmonic map \cite{Bryant84, Ejiri}. Yet it is far more than a parallel generalization. Besides that, it greatly simplifies the study of Wintgen ideal submanifolds by reducing it to surface theory. (See Theorem~\ref{thm-codim2} for stronger result in codimension two.)

As a consequence, these $m$-dimensional mean curvature spheres
is a 2-parameter family. We consider their envelope $\widehat{M}^m$, which contains $M^m$ as an open subset. The second new result is

\begin{theorem}\label{thm-envelop} For a Wintgen ideal submanifold $x:M^m\to \mathbb{S}^{m+p}$ and the envelope $\widehat{M}^m$, we have the following conclusions:

1) There is a fiber bundle structure $S^{m-2}\to\widehat{M}^m\to \overline{M}^2$ over a Riemann surface.
The fibers are all round spheres of the ambient space.

2) The projection $\pi:\widehat{M}^m\to \overline{M}^2$ is a Riemannian submersion up to a constant.

3) As a natural extension of $M^m$, $\widehat{M}^m$ is still a Wintgen ideal submanifold.
\end{theorem}

This theorem shows that Wintgen ideal submanifolds have simple and elegant structure. Based on this general picture, we can show that they arise either as cylinders, cones, rotational submanifolds, or Hopf bundles over complex curves in complex projective spaces under various specific assumptions.

This paper is organized as below. In Section~2, we will briefly review the submanifold theory in M\"obius geometry established by Changping Wang \cite{CPWang}. Section~3 gives the information on the invariants and the structure equations of Wintgen ideal submanifolds. The two results mentioned above are proved separately in Section~4 and Section~5.
Finally, we survey some recent results on Wintgen ideal submanifolds based on our joint work with Tongzhu Li and Changping Wang. These include a reduction theorem \cite{LiTZ1}, the characterization of the minimal examples \cite{XLMW}, and a classification of M\"obius homogeneous examples \cite{LiTZ3}.

\section{Submanifold theory in M\"obius geometry}
Here we follow the framework of Wang in \cite{CPWang} except that we take a different canonical lift $Y$ up to a constant.

In the classical light-cone model, the light-like directions in the Lorentz space $\mathbb{R}^{m+p+2}_1$ correspond to points in the round sphere $\mathbb{S}^{m+p}$, and the Lorentz orthogonal group correspond to conformal transformation group of $\mathbb{S}^{m+p}$. The Lorentz inner product between
$Y=(Y_0,Y_1,\cdots,Y_{m+p+1}), Z=(Z_0,Z_1,\cdots,Z_{m+p+1})\in
\mathbb{R}^{m+p+2}_1$ is
\[
\langle  Y,Z\rangle=-Y_0Z_0+Y_1Z_1+\cdots+Y_{m+p+1}Z_{m+p+1}.
\]

Let $f:M^m\rightarrow \mathbb{S}^{m+p}\subset \mathbb{R}^{m+p+1}$ be a submanifold without umbilics. Take $\{e_i|1\le i\le m\}$ as the tangent frame with respect to the induced metric $I=\d f\cdot \d f$, and $\{\theta_i\}$ as the dual 1-forms.
Let $\{n_{r}|1\le r\le p\}$ be orthonormal frame for the
normal bundle. The second fundamental form and
the mean curvature of $f$ are
\begin{equation}\label{2.1}
II=\sum\nolimits_{ij,r}h^{r}_{ij}\theta_i\otimes\theta_j
n_{r},~~H=\frac{1}{m}\sum\nolimits_{j,r}h^{r}_{jj}n_{r}=\sum\nolimits_{r}H^{r}n_{r},
\end{equation}
respectively. We define the M\"{o}bius position vector $Y:
M^m\rightarrow \mathbb{R}^{m+p+2}_1$ of $f$ by
\begin{equation}\label{2.2}
Y=\rho(1,f),~~~
~~\rho^2=\frac{1}{4}\big|II-\frac{1}{m} tr(II)I\big|^2
\end{equation}
which is a canonical lift of $f$.
Two submanifolds $f,\bar{f}: M^m\rightarrow \mathbb{S}^{m+p}$
are M\"{o}bius equivalent if there exists $T$ in the Lorentz group
$O(m+p+1,1)$ in $\mathbb{R}^{m+p+2}_1$ such that $\bar{Y}=YT.$
It follows immediately that
\begin{equation}\label{g}
\mathrm{g}=\langle \d Y,\d Y\rangle=\rho^2 \d f\cdot \d f
\end{equation}
is a M\"{o}bius invariant, called the M\"{o}bius metric of $x$.

Let $\Delta$ be the Laplacian with respect to $\mathrm{g}$. Define
\begin{equation}
N=-\frac{1}{m}\Delta Y-\frac{1}{2m^2}
\langle \Delta Y,\Delta Y\rangle Y,
\end{equation}
Let $\{E_1,\cdots,E_m\}$ be a local orthonormal frame for $(M^m,\mathrm{g})$
with dual 1-forms $\{\omega_1,\cdots,\omega_m\}$. We define tangent frame $Y_j=E_j(Y)$ and normal frame
\[
\xi_r=(H^r,n_r+H^r f).
\]
Then $\{Y,N,Y_j,\xi_{r}\}$ is a moving frame of $\mathbb{R}^{m+p+2}_1$ along $M^m$, which is orthonormal except
\[
\langle Y,Y\rangle=0=\langle N,N\rangle, ~~
\langle N,Y\rangle=1~.
\]
\begin{remark}\label{rem-xi}
Geometrically, at one point $x\in M^m$, $\xi_r$ (for any given $r$) corresponds to the unique hypersphere tangent to $M_m$ with normal vector $n_r$ and mean curvature $H^r(x)$. In particular, the spacelike subspace $\mathrm{Span}_{\mathbb{R}}\{\xi_1,\cdots,\xi_p\}$ represents a unique $m$-dimensional sphere tangent to $M^m$ with the same mean curvature vector $\sum\nolimits_r H^r n_r$. This well-defined object was naturally named \emph{the mean curvature sphere} of $M^m$ at $x$, which is well-known to share the same mean curvature at $x$ even when the ambient space is endowed with any other conformal metric.
\end{remark}
We fix the range of indices in this section as below: $1\leq
i,j,k\leq m; 1\leq r,s\leq p$. The structure equations are:
\begin{equation}\label{eq-structure}
\begin{split}
&\d Y=\sum\nolimits_i \omega_i Y_i,\\
&\d N=\sum\nolimits_{ij}A_{ij}\omega_i Y_j+\sum\nolimits_{i,r} C^r_i\omega_i \xi_{r},\\
&\d Y_i=-\sum\nolimits_j A_{ij}\omega_j Y-\omega_i N+\sum\nolimits_j\omega_{ij}Y_j
+\sum\nolimits_{j,r} B^{r}_{ij}\omega_j \xi_{r},\\
&\d \xi_{r}=-\sum\nolimits_i C^{r}_i\omega_i Y-\sum\nolimits_{i,j}\omega_i
B^{r}_{ij}Y_j +\sum\nolimits_{s} \theta_{rs}\xi_{s},
\end{split}
\end{equation}
where $\omega_{ij}$ are the connection $1$-forms of the M\"{o}bius
metric $\mathrm{g}$; $\theta_{rs}$ are the normal connection $1$-forms. The tensors
\begin{equation}
{\bf A}=\sum\nolimits_{i,j}A_{ij}\omega_i\otimes\omega_j,~~ {\bf
B}=\sum\nolimits_{i,j,r}B^{r}_{ij}\omega_i\otimes\omega_j \xi_{r},~~
\Phi=\sum\nolimits_{j,r}C^{r}_j\omega_j \xi_{r}
\end{equation}
are called the Blaschke tensor, the M\"{o}bius second fundamental
form and the M\"{o}bius form of $f$, respectively \cite{CPWang}.
The integrability conditions for the structure equations are given as below:
\begin{eqnarray}
&&A_{ij,k}-A_{ik,j}=\sum\nolimits_{r}(B^{r}_{ik}C^{r}_j
-B^{r}_{ij}C^{r}_k),\label{equa1}\\
&&C^{r}_{i,j}-C^{r}_{j,i}=\sum\nolimits_k(B^{r}_{ik}A_{kj}
-B^{r}_{jk}A_{ki}),\label{equa2}\\
&&B^{r}_{ij,k}-B^{r}_{ik,j}=\delta_{ij}C^{r}_k
-\delta_{ik}C^{r}_j,\label{equa3}\\
&&R_{ijkl}=\sum\nolimits_{r}(B^{r}_{ik}B^{r}_{jl}-B^{r}_{il}B^{r}_{jk})
+\delta_{ik}A_{jl}+\delta_{jl}A_{ik}
-\delta_{il}A_{jk}-\delta_{jk}A_{il},\label{equa4}\\
&&R^{\perp}_{rs ij}=\sum\nolimits_k
(B^{r}_{ik}B^{s}_{kj}-B^{s}_{ik}B^{r}_{kj}). \label{equa5}
\end{eqnarray}
Here the covariant derivatives $A_{ij,k}, B^{r}_{ij,k}, C^{r}_{i,j}$ are defined as usual; $R, R^{\perp}$ denote the the curvature tensor of $\mathrm{g}$ and the normal curvature tensor, respectively. The tensor $\bf B$ satisfies the following identities:
\begin{equation}
\sum\nolimits_j B^{r}_{jj}=0, ~~~\sum\nolimits_{i,j,r}(B^{r}_{ij})^2=4. \label{equa7}
\end{equation}
All coefficients in the structure equations are determined by $\{\mathrm{g}, {\bf B}\}$
and the normal connection $\{\theta_{\alpha\beta}\}$.
In particular these are the complete set of M\"obius invariants.

\section{Invariants of a Wintgen ideal submanifold}

Let $f: M^{m}\to \mathbb{S}^{m+p}$ be a Wintgen ideal submanifold. We will always assume that it is umbilic-free unless it is stated otherwise. In terms of the M\"obius invariants, that means the existence of a suitable tangent frame $\{E_1,\cdots,E_m\}$ and normal frame $\{\xi_1,\cdots,\xi_p\}$ so that the M\"obius second fundamental form are given by
\begin{equation}\label{3.1}
B^{1}=
\begin{pmatrix}
0 & 1 & 0 & \cdots & 0\\
1 & 0 & 0 & \cdots & 0\\
0  & 0 & 0 & \cdots & 0\\
\vdots & \vdots & \vdots & \ddots & \vdots\\
0  & 0 & 0 & \cdots & 0
\end{pmatrix},~~
B^{2}=
\begin{pmatrix}
1 & 0 & 0 & \cdots & 0\\
0 & -1 & 0 & \cdots & 0\\
0  & 0 & 0 & \cdots & 0\\
\vdots & \vdots & \vdots & \ddots & \vdots\\
0  & 0 & 0 & \cdots & 0
\end{pmatrix},~~
B^{\alpha}=0,~\alpha\ge 3.
\end{equation}
\begin{remark}
The reader is warned that the lift $Y$ here is different from \cite{CPWang}. Hence in the formulas below, we have removed the annoying factor $\mu=\sqrt{\frac{m-1}{4m}}$ appearing in \cite{LiTZ1,LiTZ2,LiTZ3,XLMW}.
\end{remark}
\begin{remark} The \emph{canonical distribution} $\mathbb{D}_2=\mathrm{Span}\{E_1,E_2\}$ and the normal sub-bundle $\mathrm{Span}\{\xi_1,\xi_2\}$ are well-defined if \eqref{3.1} holds and we fix our frame up to rotations
\begin{equation}\label{transform}
(\widetilde{E}_1,\widetilde{E}_2)=(E_1,E_2)
\begin{pmatrix} ~~\cos t &\sin t \\ -\sin t & \cos t\end{pmatrix},
~~(\widetilde{\xi}_1,\widetilde{\xi}_2)=(\xi_1,\xi_2)
\begin{pmatrix} \cos 2t &-\sin 2t \\ \sin 2t & ~~\cos 2t\end{pmatrix}.
\end{equation}
\end{remark}
We will adopt the convention below on the range of indices:
\[
1\le i,j,k,l\le m,~~3\le a,b\le m; ~~~~1\le r,s\le p,~~3\le \alpha,\beta\le p.
\]
By definition, we compute the covariant derivatives of $B^r_{ij}$ and obtain
\begin{equation}\label{bb1}
B^{r}_{ab,i}=0, ~~B^{\alpha}_{1a,i}=B^{\alpha}_{2a,i}=0,
\end{equation}
\begin{equation}\label{Bmu}
B^1_{12,i}=B^1_{21,i}=0,~~B^2_{11,i}=B^2_{22,i}=0,
\end{equation}
\begin{equation}\label{bb2}
\omega_{2a}=\sum\nolimits_i B^1_{1a,i}\omega_i
=-\sum\nolimits_i B^2_{2a,i}\omega_i,~~ \omega_{1a}=\sum\nolimits_i B^1_{2a,i}\omega_i
=\sum\nolimits_i B^2_{1a,i}\omega_i,
\end{equation}
\begin{equation}\label{bb3}
2\o_{12}+\theta_{12}=\sum\nolimits_i\frac{-B^1_{11,i}}{\mu}\omega_i
=\sum\nolimits_i B^1_{22,i} \omega_i
=\sum\nolimits_i B^2_{12,i} \omega_i,
\end{equation}
\begin{equation}\label{bb4}
\theta_{1\alpha}=\sum\nolimits_i B^\alpha_{12,i} \o_i,
~~\theta_{2\alpha}=\sum\nolimits_i B^\alpha_{11,i} \o_i.
\end{equation}
By \eqref{equa3}, $B^r_{ij,k}$ is symmetric for distinctive $i,j,k$. It follows from \eqref{bb1}$\sim$\eqref{bb3} that
\begin{equation*}
\begin{split}
&\o_{1a}(E_b)=B^1_{2a,b}=B^1_{ab,2}=0,
~~~~\o_{2a}(E_b)=B^1_{1a,b}=B^1_{ab,1}=0~~~(a\ne b);
\\
&\o_{1a}(E_1)=B^2_{1a,1}=B^1_{2a,1}=B^1_{21,a}=0,
~~~~\o_{2a}(E_2)=-B^2_{2a,2}=B^1_{1a,2}=B^1_{21,a}=0;\\
&B^2_{1a,2}=\mu\o_{1a}(E_2)=-\mu\o_{2a}(E_1)
=\mu(2\o_{12}+\theta_{12})(E_a)
=B^1_{2a,2}=B^1_{22,a}=-B^1_{11,a}.
\end{split}
\end{equation*}
Based on these information, we use \eqref{equa3} to compute $C^r_{i,j}$ as below:
\begin{align}
&C^1_1=B^1_{22,1}-B^1_{21,2}=B^1_{22,1},~ && C^1_2=B^1_{11,2}-B^1_{12,1}=B^1_{11,2},\label{c0}\\
&C^1_1=B^1_{aa,1}-B^1_{1a,a}=-B^1_{1a,a},~ && C^1_2=B^1_{aa,2}-B^1_{2a,a}=-B^1_{2a,a},\label{c1}\\
&C^2_1=B^2_{aa,1}-B^2_{1a,a}=-B^2_{1a,a},~ && C^2_2=B^2_{aa,2}-B^2_{2a,a}=-B^2_{2a,a},\label{c2}\\
&C^1_{a}=B^1_{22,a}-B^1_{2a,2}=0,~&&C^2_a=B^2_{11,a}-B^2_{1a,1}=0,\label{c3}\\
&C^{\alpha}_1=B^{\alpha}_{aa,1}-B^{\alpha}_{a1,a}=0,~&&
C^{\alpha}_2=B^{\alpha}_{aa,2}-B^{\alpha}_{a2,a}=0,\label{c4}\\
&C^{\alpha}_a=B^{\alpha}_{11,a}-B^{\alpha}_{1a,1}=B^{\alpha}_{11,a},~
&&C^{\alpha}_a=B^{\alpha}_{22,a}-B^{\alpha}_{2a,2}=B^{\alpha}_{22,a}.
~~(\forall~a,\alpha) \label{c5}
\end{align}
Utilizing the fact $\sum\nolimits_iB^{\alpha}_{ii,k}=0$, we deduce from \eqref{bb1} that $C^{\alpha}_a=0$.
By \eqref{Bmu}, \eqref{bb2} and \eqref{c0}$\sim$\eqref{c5}, the final result is
\begin{align}
&C^1_1=-C^2_2=-\omega_{2a}(e_a),~~C^1_2=C^2_1=-\o_{1a}(e_a), \label{C1C2}\\
&C^1_a=C^2_a=0, ~~C^{\alpha}_i=0\label{Cai}.
\end{align}
For similar reasons, \eqref{c4} and \eqref{c5} imply
\begin{align*}
\theta_{1\alpha}(E_1)-\theta_{2\alpha}(E_2)&
=B^\alpha_{12,1}-B^\alpha_{11,2}
=-C^\alpha_2=0,\\
\theta_{1\alpha}(E_2)+\theta_{2\alpha}(E_1)&
=(B^\alpha_{21,2}-B^\alpha_{22,1})+(B^\alpha_{22,1}+B^\alpha_{11,1})
=-C^\alpha_1=0.
\end{align*}
We summarize the most important information on the connection 1-forms as below:
\begin{proposition}\label{prop-connection}
For a Wintgen ideal submanifold, denote
\begin{equation}\label{eq-LUV}
L_a=-B^1_{11,a},~V=C^1_2=C^2_1, ~U=C^2_2=-C^1_1,
~S_\alpha=B^\alpha_{11,2},~T_\alpha =B^\alpha_{11,1}.
\end{equation}
We can choose a suitable frame $\{E_3,\cdots,E_m\}$ so that $L_a=-B^1_{11,a}=0$ when $a\ge 4$ and denote $L\triangleq L_3=-B^1_{11,3}$. Then
\begin{equation}\label{omega1a2a}
\omega_{1a}
=L_a\omega_2-V\omega_a,~~~~
\omega_{2a}
=-L_a\omega_1+U\omega_a;
\end{equation}
\begin{equation}\label{theta}
2\o_{12}+\theta_{12}=-U\o_1
-V\o_2+L\o_3;
\end{equation}
\begin{equation}\label{theta1a2a}
\theta_{1\alpha}
=S_\alpha\o_1-T_\alpha\o_2,~~~~
\theta_{2\alpha}
=T_\alpha\o_1+S_\alpha\o_2.
\end{equation}
\end{proposition}

Before discussing the properties of the conformal Gauss map $\Xi=\xi_1\wedge\cdots\wedge\xi_p$ in the next section, we notice that
the subspace $\mathrm{Span}\{\xi_1,\xi_2\}$ also defines a map into the Grassmannian $\mathrm{Gr}(2,\mathbb{R}^{m+p+2}_1)$. This is also represented by $[\xi_1-\i\xi_2]$ in a complex quadric
\[
\mathbb{Q}^{m+p}_+=\{[Z]\in \mathbb{C}P^{m+p+1}|~
Z\in\mathbb{R}^{m+4}_1\otimes\mathbb{C},\<Z,Z\>=0,\<Z,\bar{Z}\> >0\}.
\]
We denote $\xi=\xi_1-\i\xi_2$, and call $[\xi]$ \emph{the second Gauss map} of the Wintgen ideal submanifold.
When the codimension $p=2$, $[\xi]$ is equivalent to the conformal Gauss map $\Xi$. To understand its geometry, substitute \eqref{3.1}, \eqref{C1C2}, \eqref{Cai} and \eqref{theta1a2a} into the last structure equation of \eqref{eq-structure}. The result is
\begin{equation}\label{J}
\d(\xi_{1}-\i\xi_{2})=\i(\o_1+\i\o_2)(\eta_1+\i\eta_2)
+\i\theta_{12}(\xi_1-\i\xi_2)+(\o_1-\i\o_2)\cdot
\sum\nolimits_{\alpha}(S^\alpha-\i T^\alpha)\xi_\alpha,
\end{equation}
where
\begin{equation}\label{eq-eta12}
\eta_1=Y_1+ C_2^1 Y=Y_1+VY,~~
\eta_2=Y_2+ C^1_1 Y=Y_2-UY.
\end{equation}
This indicate that the image of $[\xi]$ degenerates to a 2-dimensional surface, a property also shared by the conformal Gauss map $\Xi$.

Differentiate once more, the result would be
\begin{equation}\label{eq-deta}
%\begin{split}
\d(\eta_1+\i\eta_2)=(\o_1+\i\o_2)\Big[-\tilde{Y}-FY
+\big(\frac{G}{L}-\i L\big)\eta_3\Big]
-\i\Omega_{12}(\eta_1+\i\eta_2)+\i(\o_1-\i\o_2)(\xi_1-\i\xi_2),
%\end{split}
\end{equation}
where $\Omega_{12}=\<\d\eta_1,\eta_2\>$ is a connection 1-form,
\begin{equation}
F=A_{11}-C^1_{2,1}+\frac{1}{2}
\left(U^2+V^2-\left(\frac{G}{L}\right)^2\right),
~G=A_{12}-C^1_{2,2}=(C^1_{1,1}-C^1_{2,2})/2;
\end{equation}
\begin{equation}\label{eq-ytilde}
\tilde{Y}=N-VY_1+UY_2+\frac{G}{L}Y_3-\frac{1}{2}
\left(U^2+V^2+\left(\frac{G}{L}\right)^2\right)Y, ~~\eta_3=Y_3-\frac{G}{L}Y.
\end{equation}
Note that we have assumed $L\ne 0$ at here. To prove \eqref{eq-deta}, we have used \eqref{equa2} to compute $A_{1j}$. We omit the straightforward yet tedious computation at here.

\section{The conformal Gauss map as a harmonic map}

\begin{proposition}\label{prop-submersion}
For an umbilic-free Wintgen ideal submanifold $f:M^m\to \mathbb{S}^{m+p}$, the following three conclusions hold true:

(1) The image of the conformal Gauss map $
\Xi=\xi_1\wedge\cdots\wedge\xi_p:M^m\to \mathrm{Gr}(p,\mathbb{R}^{m+p+2}_1)$
is a real 2-dimensional surface $\overline{M}^2$.

(2) The projection $\pi:M^m\to \overline{M}^2$ determined by $\Xi$ is a Riemannian submersion (up to the factor $\sqrt{2}$),
where $M^m$ is endowed with the M\"obius metric and $\overline{M}^2\subset\mathrm{Gr}(p,\mathbb{R}^{m+p+2}_1)$ with the induced metric.

(3) The distribution $\mathbb{D}_2^{\bot}=\mathrm{Span}\{E_3,\cdots,E_m\}$ is integrable. Its integral submanifolds are exactly the fibers of the submersion mentioned above.
\end{proposition}
\begin{proof}
When $p=2$, these conclusions and Theorem~\ref{thm-harmonic} has been proved in \cite{LiTZ2}.
In the general case when $p\ge 3$, we adopt the convention $3\le a\le m, 3\le \alpha\le p$ on the indices.
Then it follows from \eqref{eq-structure} and Proposition~\ref{prop-connection} that
\begin{align}
E_1(\Xi)&=-[\eta_2\wedge\xi_2\wedge\!\ast
+\xi_1\wedge\eta_1\wedge\!\ast], ~~(\!\ast\triangleq\xi_3\wedge\xi_4\wedge\cdots\wedge\xi_p)\\
E_2(\Xi)&=-[\eta_1\wedge\xi_2\wedge\!\ast
-\xi_1\wedge\eta_2\wedge\!\ast],\\
E_a(\Xi)&=0,~~~\forall~3\le a\le m.
\end{align}
Consequently, the tangent space $\Xi_*T_x\overline{M}^2\subset T_{\Xi(x)}\mathrm{Gr}(p,\mathbb{R}^{m+p+2}_1)$ is a plane given by
\[
\mathrm{Span}\{\eta_2\wedge\xi_2\wedge\!\ast
+\xi_1\wedge\eta_1\wedge\!\ast, ~~ \eta_1\wedge\xi_2\wedge\!\ast
-\xi_1\wedge\eta_2\wedge\!\ast\} ,
\]
and the induced metric is $\d s^2=2[(\omega_1)^2+(\omega_2)^2]$.
This proves the first two conclusions. In particular the image of $\Xi$ is a 2-dimensional surface $\overline{M}^2$.

As the the kernel of the tangent map $\pi_*$,   $\mathbb{D}_2^\bot$, the vertical subspace at every point, is always an integrable distribution whose integral submanifolds are nothing but the fibers of this submersion. Conclusion (3) follows immediately (or by the expressions of $\omega_{1\alpha},\omega_{2\alpha}$ in \eqref{omega1a2a} and the Frobenius Theorem).
\end{proof}

\begin{proof}[Proof to Theorem~\ref{thm-harmonic}]~\\
\indent According to Proposition~\ref{prop-submersion}, we can regard $\Xi$ as a conformal immersion from the Riemann surface $\overline{M}^2$ to $\mathrm{Gr}(p,\mathbb{R}^{m+p+2}_1)$. $E_1,E_2$ can be viewed as horizontal lift of an orthonormal basis (up to the factor $\sqrt{2}$) of $(T\overline{M}^2,\d s^2)$. The second fundamental form of $\Xi(\overline{M}^2)$ can be read out from a straightforward computation as below using the structure equations:
\begin{align*}
E_1&E_1(\Xi)
=2\Xi+(\Omega_{12}+\theta_{12})(E_1)
\left[\xi_1\wedge\eta_2\wedge\!\ast
-\eta_1\wedge\xi_2\wedge\!\ast\right]
+2\eta_1\wedge\eta_2\wedge\!\ast\\
&-L\eta_3\wedge\xi_2\wedge\!\ast
-\xi_1\wedge\left(\hat{F}Y+\hat{Y}+\frac{G}{L}\eta_3\right)\wedge\!\ast
+\xi_1\wedge\xi_2\cdots\wedge
(S_{\alpha}\eta_2+T_{\alpha}\eta_1)\wedge\cdots\xi_p.
\end{align*}
In the final expression, the first term is the radial component, the second is the tangent component, and the third term can be ignored because it is not in the tangent space $T_{\Xi}\mathrm{Gr}(p,\mathbb{R}^{m+4}_1)$ at $\Xi=\xi_1\wedge\cdots\wedge \xi_p$. The last three terms are the normal component. Similarly we compute out
\begin{align*}
E_2&E_2(\Xi)
=2\Xi+(\Omega_{12}+\theta_{12})(E_2)
\left[\eta_2\wedge\xi_2\wedge\!\ast
+\xi_1\wedge\eta_1\wedge\!\ast\right]
+2\eta_1\wedge\eta_2\wedge\!\ast\\
&+L\eta_3\wedge\xi_2\wedge\!\ast
+\xi_1\wedge\left(\hat{F}Y+\hat{Y}+\frac{G}{L}\eta_3\right)\wedge\!\ast
+\xi_1\cdots\wedge
(-T_{\alpha}\eta_1-S_{\alpha}\eta_2)\wedge\cdots\xi_p.
\end{align*}
Thus $(E_1E_1+E_2E_2)\Xi$ has only radial and tangent components. In other words, the mean curvature vector of the surface
$\Xi:\overline{M}^2\subset \mathrm{Gr}(p,\mathbb{R}^{m+4}_1)$ vanishes. In the same manner we derive
\begin{align*}
E_1&E_2(\Xi)
=(\Omega_{12}+\theta_{12})(E_1)
\left[\eta_2\wedge\xi_2\wedge\!\ast
+\xi_1\wedge\eta_1\wedge\!\ast\right]\\
&+L\xi_1\wedge\eta_3\wedge\!\ast
-\left(\hat{F}Y+\hat{Y}+\frac{G}{L}\eta_3\right)\wedge\xi_2\wedge\!\ast
+\xi_1\cdots\wedge
(S_{\alpha}\eta_1-T_{\alpha}\eta_2)\wedge\cdots\xi_p~.
\end{align*}
Its normal component has the same squared norm
as that of $E_1E_1(\Xi)$ and $E_2E_2(\Xi)$. Thus its curvature ellipse is a circle, which is the characteristic of a \emph{super-conformal} surface. So $\Xi:\overline{M}^2\subset \mathrm{Gr}(p,\mathbb{R}^{m+4}_1)$ is a conformal super-minimal immersion.
\end{proof}

\section{The spherical foliation structure}
This section is devoted to the proof of Theorem~\ref{thm-envelop}.

By Theorem~\ref{thm-harmonic}, the mean curvature spheres $\mathrm{Span}\{\xi_1,\cdots,\xi_p\}$ is a 2-parameter family, with the parameter space being $\overline{M}^2$. It is well-known that such a sphere congruence has an envelope $\widehat{M}^m$ if and only if
$V=\mathrm{Span}\{\xi_r,\d\xi_r:1\le r\le p\}$ form a space-like sub-bundle of the trivial bundle $\mathbb{R}^{m+p+2}_1\times\overline{M}^2$. This is satisfied in our situation by \eqref{J}, with $V=\mathrm{Span}\{\xi_r,\eta_1,\eta_2\}$ being a spacelike sub-bundle of rank $p+2$. In particular, the points of the envelope correspond to the lightlike directions in its orthogonal sub-bundle $V^\bot$ over $\overline{M}^2$.
By construction, $\widehat{M}^m\supset M^m$; in general we would expect it to be a $m$-dimensional submanifold (possibly with singularities).

We have noticed that the distribution $\mathbb{D}_2^{\bot}=\mathrm{Span}\{E_3,\cdots,E_m\}$
is integrable; the integral submanifolds are fibers of the Riemannian submersion mentioned before. We assert that each fiber is contained in a $(m-2)$-dimensional sphere determined by the spacelike subspace
$V$ at some point $q\in\overline{M}^2$. This is because of \eqref{eq-deta}, which implies that the subspace $V$
is fixed along any integral submanifold of $\mathbb{D}_2^{\bot}=\mathrm{Span}\{E_3,\cdots,E_m\}$. In particular, the integration of $Y$ along $\mathbb{D}_2^{\bot}$ is always contained in $V^\bot$, which implies that any integral submanifold is located on the corresponding $(m-2)$-dimensional sphere.
This proves the first conclusion of Theorem~\ref{thm-envelop}.

Next we introduce a new moving frame $\{Y,\hat{Y},\eta_1,\eta_2,\eta_a;
\xi_r\}$ along $M$, which is an orthonormal frame except that $Y,{\hat Y}$ are lightlike with $\langle Y,{\hat Y}\rangle=1$. They are
\begin{equation}\label{eq-eta}
\eta_1=Y_1+VY,~~\eta_2=Y_2-UY,~~~\eta_a=Y_a+\lambda_a Y.
\end{equation}
Here $\{\lambda_a\}_{a=3}^m$ are real numbers chosen arbitrarily, depending smoothly on the underlying Riemann surface $\overline{M}^2$.
By conclusions in the previous paragraph, $\overline{M}^2$ and $\{\lambda_a\}_{a=3}^m$ give a parametrization of $\widehat{M}^m$. When $\{\lambda_a\}_{a=3}^m$ vary arbitrarily, the point corresponding to the lightlike direction
\begin{equation}\label{eq-yhat}
\hat{Y}=N-\frac{1}{2}(V^2+U^2+\sum\nolimits_a \lambda_a^2)Y-V Y_1+U Y_2+\sum\nolimits_a\lambda_a Y_a
\end{equation}
will travel around the whole envelope $\widehat{M}^m$. Thus we may regard $\hat{Y}$ as a local lift of the parameterized submanifold $\widehat{M}^m$, and any property of $\widehat{M}^m$ can be obtained from $\hat{Y}$ with arbitrarily given $\{\lambda_a\}_{a=3}^m$. This is the key point in our analysis.

We will focus on the regular subset where $\widehat{M}^m$ is immersed. Using the new moving frame \eqref{eq-eta} and \eqref{eq-yhat}, there is a new system of structure equations:
\begin{align}
\d\xi_{1}&=-\o_2\eta_1-\o_1\eta_2+\theta_{12}\xi_2,
\label{3.3}\\
\d\xi_{2}&=-\o_1\eta_1+\o_2\eta_2-\theta_{12}\xi_1,
\label{3.4}\\
\d\xi_{\alpha}&=-\theta_{1\alpha}\xi_1-\theta_{2\alpha}\xi_2
+\sum\nolimits_\beta\theta_{\alpha\beta}\xi_\beta,
\label{3.5}\\
\d\eta_1&=-{\hat\o}_1Y-\o_1{\hat
Y}+\sum\nolimits_k\Omega_{1k}\eta_k+\o_2\xi_1+\o_1\xi_2,\label{3.6}\\
\d\eta_2&=-{\hat\o}_2Y-\o_2{\hat
Y}+\sum\nolimits_k\Omega_{2k}\eta_k+\o_1\xi_1-\o_2\xi_2,\label{3.7}\\
\d\eta_a&=-{\hat\o}_aY-\o_a{\hat Y}+\sum\nolimits_k\Omega_{ak}\eta_k,\label{3.8}\\
\d Y&=\o Y+\o_1\eta_1+\o_2\eta_2+\sum\nolimits_a\o_a\eta_a,\label{3.9}\\
\d{\hat Y}&=-\o {\hat Y}+{\hat\o}_1\eta_1+{\hat\o}_2\eta_2+\sum\nolimits_a{\hat\o}_a\eta_a. \label{3.10}
\end{align}
Here $\o,\o_k,\hat\o_k,\Omega_{jk}$ are
1-forms locally defined on $\widehat{M}^m$ which we don't need to know explicitly.

We claim that the envelope $\widehat{M}^m$, viewed as an immersion $[\hat{Y}]$ into the sphere, still has $\mathrm{Span}_{\mathbb{R}}\{\xi_1,\xi_2,\cdots,\xi_p\}$ as its mean curvature sphere.

As a preparation, it is important to notice that there exist some functions $\hat{F},\hat{G}$ such that
\begin{equation}\label{eq-FG}
{\hat\o_1}=\hat{F}\o_1+\hat{G}\o_2,~~ {\hat\o_2}=-\hat{G}\o_1+\hat{F}\o_2.
\end{equation}
This follows from \eqref{eq-deta} and \eqref{eq-yhat} directly (or from the integrability conditions of the system \eqref{3.3}$\sim$\eqref{3.10}). Based on this, under the induced metric $\langle \d\hat{Y},\d\hat{Y}\rangle=\sum_{j=1}^m\hat\o_j^2$
we take a frame $\{\hat{E}_j\}_{j=1}^m$ so that $\hat\o_i(\hat{E}_j)=(\hat{F}^2+\hat{G}^2)\delta_{ij}$.
Since $\widehat{M}^m$ is assumed to be immersed, $\hat{F}^2+\hat{G}^2\ne 0$. Modulo the components in $\mathbb{D}_2^\bot=\mathrm{Span}\{E_3,\cdots,E_m\}$ one gets
\begin{equation}\label{eq-Ehat}
\hat{E}_1\approx\hat{F}\hat{E}_1+\hat{G}\hat{E}_2,~~
\hat{E}_2\approx-\hat{G}\hat{E}_1+\hat{F}\hat{E}_2,~~
\hat{E}_a\approx0~~(\mathrm{mod}~\mathbb{D}_2^\bot).
\end{equation}
Next we compute the Laplacian $\hat{\Delta}\hat{Y}$. The mean curvature sphere at $\hat{Y}$ is determined by
\[
\mathrm{Span}_{\mathbb{R}}\{\hat{Y},\hat{Y}_j,\sum\nolimits_{j=1}^m \hat{E}_j\hat{E}_j(\hat{Y})\}
=\mathrm{Span}_{\mathbb{R}}\{\hat{Y},\hat{Y}_j, \hat{\Delta}\hat{Y}\}.
\]

To verify our claim, it suffices to show $\langle\sum_{j=1}^m \hat{E}_j\hat{E}_j(\hat{Y}),\xi_r\rangle=0.$
Because $\langle\hat{Y},\xi_r\rangle=0=\langle \d\hat{Y},\xi_r\rangle=\langle \hat{Y},\d\xi_r\rangle$, this is also equivalent to
\[
\langle \hat{Y},\sum\nolimits_{j=1}^m \hat{E}_j\hat{E}_j(\xi_r)\rangle=0,
~~~~1\le r\le p.
\]
This can be checked directly using \eqref{eq-Ehat} and \eqref{3.3}$\sim$\eqref{3.7}. As a consequence, the previous claim is proved.

Finally, for $\hat{Y}$ we take its canonical lift, whose derivatives are clearly combinations of $\hat{Y},\eta_1,\eta_2,\eta_a$. Its normal frame is just $\{\xi_1,\xi_2\}$ as we have shown. One reads from \eqref{3.3} and \eqref{3.4} that this is still a Wintgen ideal submanifold, which finishes the proof.

\section{Special classes of Wintgen ideal submanifolds}

This section reviews our recent work on Wintgen ideal submanifolds from a unified viewpoint of the conformal Gauss map $\Xi$ and the fiber bundle structure over $\overline{M}^2$. In the codimension two case we have the following result \cite{LiTZ2}, where $\Xi$ can be identified with the second Gauss map $[\xi]$ from the Riemann surface $\overline{M}^2$. The theorem below is stronger than Theorem~\ref{thm-harmonic} by replacing \emph{harmonic map} by \emph{holomorphic map}. It also supplement Theorem~\ref{thm-envelop} by showing the converse is also true.

\begin{theorem}\label{thm-codim2}\cite{LiTZ2}
The conformal Gauss map $[\xi]=[\xi_{1}-i\xi_{2}] \in \mathbb{Q}^{m+2}_+$ of a Wintgen ideal submanifold of codimension two is a holomorphic and 1-isotropic curve, i.e., with respect to a local complex coordinate $z$ of $\overline{M}^2$, $\xi_{\bar{z}}\parallel \xi, \<\xi_z,\xi_z\>=0.$
Conversely, given a holomorphic 1-isotropic curve
$[\xi]:\overline{M}^2\to \mathbb{Q}^{m+2}_+\subset\mathbb{C}P^{m+3},$
the envelope $\widehat{M}^m$ of the corresponding 2-parameter family spheres is a $m$-dimensional Wintgen ideal submanifold (at the regular points).
\end{theorem}

\begin{remark}
Dajczer et. al. \cite{Dajczer3} have shown that codimension two Wintgen ideal submanifolds can always be constructed from Euclidean minimal surfaces. Our description is equivalent to theirs by a complex stereographic projection from $\mathbb{Q}^{m+2}_+$ to the complex space $\mathbb{C}^{m+2}=\mathbb{R}^{m+2}\otimes\mathbb{C}$, which maps holomorphic 1-isotropic curves in one space to holomorphic 1-isotropic curves in another space.
\end{remark}

Consider the canonical distribution $\mathbb{D}_2=\mathrm{Span}\{E_1,E_2\}$. In the Riemannian submersion structure $\pi:\widehat{M}^m\to \overline{M}^2$, it can be viewed as the horizontal lift (at various points) of the tangent plane $T\overline{M}^2$. By Proposition~\ref{prop-connection}, $\mathbb{D}_2$ is integrable if and only if $L=0$. This is the geometric meaning of the invariant $L=-B^1_{11,3}$ for a Wintgen ideal submanifold. In general we may consider the integrable distribution generated by $\mathbb{D}_2$ with the lowest dimension $k$ and denote it as $\mathbb{D}$. Related with the case $k<m$ we have the following conjecture, which has been proved for $k=2$ \cite{LiTZ1} and for $k=3,4,5$ (not published).

\begin{conjecture}\label{conj-reduce}
Let $x:M^m\longrightarrow \mathbb{R}^{m+p}$ be a Wintgen ideal submanifold without umbilic points. If the canonical distribution $\mathbb{D}_2$ generates an integrable distribution $\mathbb{D}$ with dimension $k<m$, then locally $x$ is M\"{o}bius equivalent to a cone (res. a cylinder; a rotational submanifold) over a $k$-dimensional minimal Wintgen ideal submanifold in ${\mathbb S}^{k+p}$ (res. in ${\mathbb R}^{k+p}$; in ${\mathbb H}^{k+p}$.)
\end{conjecture}

In our attempts to prove this \emph{reduction conjecture} for Wintgen ideal submanifolds with a low dimensional ($\mathrm{dim}(\mathbb{D})=k$ is fixed) integrable distribution $\mathbb{D}$, we notice that it is possible to choose a new frame
$\{Y,\hat{Y},\eta_1,\eta_2,\eta_a\}$ with similar expressions as \eqref{eq-eta} and \eqref{eq-yhat} (some kind of \emph{gauge transformation}), which helps to find a decomposition of $\mathbb{R}^{m+p+2}_1$ into invariant subspaces \cite{LiTZ1}. Moreover, the integrability of $\mathbb{D}$ implies that the Lorentz plane bundle $\mathrm{Span}\{Y,\hat{Y}\}$ is flat, i.e., the connection 1-form $\o=\d Y\cdot \hat{Y}$ is closed. Another conclusion is that the correspondence $[Y]\leftrightarrow [\hat{Y}]$ is a conformal map from $\widehat{M}^m$ to itself. We strongly believe that these facts are always true for arbitrary $k\ge 2$.

In all cases we know, $\o$ is a well-defined M\"obius invariant whose explicit expression depends on $k$. For example, when $k=3$,
$\o=-C^1_2\o_1-C^1_1\o_2+\frac{E_3(L)}{L}\o_3$ \cite{XLMW}.

A natural question arises: for a fixed $k$ and Wintgen ideal submanifolds of dimension $m=k$ which are \emph{irreducible (i.e., the only integrable distribution containing $\mathbb{D}_2$ is the tangent bundle of $M$)}, what is the meaning of $\d\o=0$? We conjecture the following characterization result, which has been proved for the case $m=3,p=2$ \cite{XLMW} and the general 3-dimensional case (to appear later).

\begin{conjecture}\label{conj-minimal}
For an irreducible Wintgen ideal submanifold $M^k$ of dimension $k\ge 3$, if $\d\o=0$, then $M^k$ is M\"obius equivalent to a minimal Wintgen ideal submanifold in either of the three space forms.
\end{conjecture}

A main difficulty in proving these two conjectures for arbitrary dimension $k$ is that when $k$ changes we have to modify the frame $\{Y,\hat{Y},\eta_1,\eta_2,\eta_a\}$ as well as the expression $\o$ accordingly, and a unified treatment is still lacking.

Finally, we mention that under the condition of being M\"obius homogeneous, Wintgen ideal submanifolds have been classified \cite{LiTZ3}. Not surprisingly they come from famous examples of homogeneous minimal surfaces.

\begin{theorem}\label{thm-homog}\cite{LiTZ3}
A M\"{o}bius homogeneous Wintgen ideal submanifold is M\"{o}bius equivalent to either an affine subspace in $\mathbb{R}^{m+p}$, or a cone over a Veronese surface in $S^{2k}$,
or a cone over a Clifford type flat minimal surface in $S^{2k+1}$, or a cone over $\pi^{-1}\circ f: \mathbb{C}P^1\to S^{2k+1}$ where $f: \mathbb{C}P^1\to\mathbb{C}P^k$ is the veronese mapping and $\pi$ is the Hopf bundle projection.
\end{theorem}

It is interesting to note that for a M\"obius homogeneous Wintgen ideal submanifold $M$, the M\"obius form must vanish, and $M$ can always be reduced to 2 or 3 dimensional minimal examples in the sense of Conjecture~\ref{conj-reduce}. Proving these facts are the key steps in obtaining the final classification in \cite{LiTZ3}.

\end{document}